\documentclass[11pt]{article}
\pdfoutput=1

\usepackage[margin=1.3in]{geometry}

\usepackage{graphicx}
\usepackage{amssymb,amsmath}
\usepackage{amsthm,enumerate}
\usepackage{hyperref,xcolor}
\usepackage{mathrsfs}
\usepackage{extarrows}
\usepackage{textcomp}

\makeatletter

\newdimen\bibspace
\makeatother

\numberwithin{equation}{section}

\newtheorem{theorem}{Theorem}[section]
\newtheorem{lemma}[theorem]{Lemma}
\newtheorem{proposition}[theorem]{Proposition}

\newtheorem{remark}[theorem]{Remark}

\def\XXint#1#2#3{{\setbox0=\hbox{$#1{#2#3}{\int}$ }
\vcenter{\hbox{$#2#3$ }}\kern-.6\wd0}}

\begin{document}

\title{Symmetry of solutions of minimal gradient graph equations on punctured space}

\author{Zixiao Liu,\quad Jiguang Bao\footnote{Supported in part by National Natural Science Foundation of China (11871102 and 11631002).}}
\date{\today}

\maketitle

\begin{abstract}
  In this paper, we study  symmetry and existence of solutions of minimal gradient graph equations
   on punctured space $\mathbb R^n\setminus\{0\}$, which include the Monge-Amp\`ere equation, inverse harmonic Hessian equation and the special Lagrangian equation.
  This  extends the classification results of Monge-Amp\`ere equations.
  Under some conditions, we also give the characterization of the solvability on exterior Dirichlet problem in terms of their asymptotic behaviors.
\end{abstract}

\section{Introduction}

We consider the following fully nonlinear elliptic equations
\begin{equation}\label{equ:SPL}
  F_{\tau}(\lambda(D^2u))=C_0\quad\text{in }\mathbb R^n\setminus\{0\},
\end{equation}
where  $C_0$ is a constant,  $\lambda(D^2u)=(\lambda_1,\lambda_2,\cdots,\lambda_n)$ are $n$ eigenvalues of  Hessian matrix $D^2u$, $\tau\in [0,\frac{\pi}{2}]$ and  $$
F_{\tau}(\lambda):=\left\{
\begin{array}{ccc}
\displaystyle  \frac{1}{n} \sum_{i=1}^{n} \ln \lambda_{i}, & \tau=0,\\
\displaystyle  \frac{\sqrt{a^{2}+1}}{2 b} \sum_{i=1}^{n} \ln \frac{\lambda_{i}+a-b}{\lambda_{i}+a+b},
  & 0<\tau<\frac{\pi}{4},\\
  \displaystyle-\sqrt{2} \sum_{i=1}^{n} \frac{1}{1+\lambda_{i}}, & \tau=\frac{\pi}{4},\\
  \displaystyle\frac{\sqrt{a^{2}+1}}{b} \sum_{i=1}^{n} \arctan \displaystyle\frac{\lambda_{i}+a-b}{\lambda_{i}+a+b}, &
  \frac{\pi}{4}<\tau<\frac{\pi}{2},\\
  \displaystyle\sum_{i=1}^{n} \arctan \lambda_{i}, & \tau=\frac{\pi}{2},\\
\end{array}
\right.
$$
$a=\cot \tau, b=\sqrt{\left|\cot ^{2} \tau-1\right|}$. Equations \eqref{equ:SPL} origin from gradient graph $(x, D u(x))$ with zero mean curvature under various metrics equipped on $\mathbb R^n\times\mathbb R^n$.
In 2010, Warren \cite{Warren} first proved that if $u\in C^2(\Omega)$ is a solution of $F_{\tau}(\lambda(D^2u))=C_0$ in $\Omega\subset\mathbb R^n$, then the volume of $(x,Du(x))$ is a maximal for $\tau\in [0,\frac{\pi}{4})$ and minimal for $\tau\in(\frac{\pi}{4},\frac{\pi}{2}]$ among all homologous, $C^1$, space-like $n$-surfaces in $(\mathbb R^n\times\mathbb R^n,g_{\tau})$, where
\begin{equation*}
g_{\tau}=\sin \tau \delta_{0}+\cos \tau g_{0},\quad\tau \in\left[0, \frac{\pi}{2}\right],
\end{equation*}
is the linearly combined metric of standard Euclidean metric
\begin{equation*}
\delta_{0}=\sum_{i=1}^{n} d x_{i} \otimes d x_{i}+\sum_{j=1}^{n} d y_{j} \otimes d y_{j},
\end{equation*}
and the pseudo-Euclidean metric
\begin{equation*}
g_{0}=\sum_{i=1}^{n} d x_{i} \otimes d y_{i}+ \sum_{j=1}^{n} d y_{j} \otimes d x_{j}.
\end{equation*}

If $\tau=0$, then \eqref{equ:SPL} becomes the Monge-Amp\`ere equation $$
\det D^2u=e^{nC_0}.
$$
For the Monge-Amp\`ere equation on entire $\mathbb R^n$, the classical theorem by J\"orgens \cite{Jorgens}, Calabi \cite{Calabi} and Pogorelov \cite{Pogorelov} states that any convex classical solution of $\operatorname{det}D^{2} u=1$ in $\mathbb{R}^{n}$ must be a quadratic polynomial. See Cheng-Yau \cite{ChengandYau}, Caffarelli \cite{7} and Jost-Xin \cite{JostandXin} for different proofs and extensions.
For the Monge-Amp\`ere equation in exterior domain of $\mathbb R^n$, there are exterior J\"orgens-Calabi-Pogorelov type results by Ferrer-Mart\'{\i}nez-Mil\'{a}n \cite{FMM99} for $n=2$ and Caffarelli-Li \cite{CL}, which state that any  locally  convex solution must be asymptotic to quadratic polynomials (for $n=2$ we need additional $\ln$-term) near infinity. For the Monge-Amp\`ere equation on punctured space $\mathbb R^n\setminus\{0\}$, there are classification results
by J\"orgens \cite{Jorgens} for $n=2$ and Jin-Xiong \cite{Jin-Xiong}
 that shows every  locally convex function of $\det D^2u=1$, modulo an affine transform, has to be $u(x)=\int_{0}^{|x|}\left(r^{n}+c\right)^{\frac{1}{n}} \mathrm{~d} r$ for some $c\geq 0$. For further generalization and refined asymptotics on J\"orgens-Calabi-Pogorelov type results we refer to \cite{Bao-Li-Zhang-CVPDE2015,Peroidic_MA,RemarkMA-2020,Peroidic_MA2} and the references therein.

If $\tau=\frac{\pi}{2}$, then \eqref{equ:SPL} becomes the special Lagrangian equation
\begin{equation}\label{equ-temp-spl}
\sum_{i=1}^{n} \arctan \lambda_{i}\left(D^{2} u\right)=C_0.
\end{equation}
For special Lagrangian equation on entire $\mathbb R^n$, there are Bernstein-type results by Yuan \cite{Yu.Yuan1,Yu.Yuan2}, which state that any classical solution of \eqref{equ-temp-spl}  and
\begin{equation}\label{equ-cond-spl}
   D^2u\geq \left\{
  \begin{array}{lll}
    -KI, & n\leq 4,\\
    -(\frac{1}{\sqrt 3}+\epsilon(n))I, & n\geq 5,\\
  \end{array}
  \right.\quad\text{or}\quad C_0>\frac{n-2}{2}\pi,
\end{equation}
must be a quadratic polynomial, where  $I$ denote the identity matrix, $K$ is any constant and $\epsilon(n)$ is a small dimensional constant. For special Lagrangian equation on exterior domain of $\mathbb R^n$, there is an exterior Bernstein-type result by Li-Li-Yuan \cite{Li-Li-Yuan-Bernstein-SPL}, which states that any classical solution of \eqref{equ-temp-spl} on exterior domain with \eqref{equ-cond-spl}  must be asymptotic to quadratic polynomial  (for $n=2$ we need additional $\ln$-term) near infinity. Furthermore,  Chen-Shankar-Yuan \cite{chen2019regularity} proved that a convex viscosity solution of \eqref{equ-temp-spl} must be smooth.

If $\tau=\frac{\pi}{4}$, then \eqref{equ:SPL} is a translated inverse harmonic Hessian equation $\sum_{i=1}^n\frac{1}{\lambda_i(D^2u)}=1$, which is a special form of Hessian quotient equation. For Hessian quotient equation on entire $\mathbb R^n$, there is a Bernstein-type result  by Bao-Chen-Guan-Ji \cite{BCGJ}.

For general $\tau\in[0,\frac{\pi}{2}]$, Warren \cite{Warren} proved the Bernstein-type results under suitable semi-convex conditions by the results of J\"orgens \cite{Jorgens}-Calabi \cite{Calabi}-Pogorelov \cite{Pogorelov}, Flanders \cite{Flanders} and Yuan \cite{Yu.Yuan1,Yu.Yuan2}. In our earlier work \cite{bao-liu2020asymptotic}, we generalized the results to equation \eqref{equ:SPL} on exterior domain and provide finer asymptotic expansions. For further generalization with perturbed right hand side we refer to \cite{bao-liu2021asymptotic} etc.

In this paper, we focus on $n\geq 3$ and prove a type of classification results as in J\"orgens \cite{Jorgens} and Jin-Xiong \cite{Jin-Xiong}.

Hereinafter, we let $x^T$ denote the transpose of vector $x\in\mathbb R^n$, $O^T$ denote the transpose of $n\times n$ matrix $O$,  $\mathtt{Sym}(n)$ denote the set of symmetric $n\times n$ matrix, $\sigma_k(\lambda)$ denote the $k$th elementary form of $\lambda$  and
$DF_{\tau}(\lambda(A))$ denote the matrix with elements being value of partial derivative of $F_{\tau}(\lambda(M))$ w.r.t $M_{ij}$ variable at matrix $A$.
A function $u(x)$ is called generalized symmetric with respect to $A\in \mathtt{Sym}(n)$ if it relies only on the value of
$x^TAx$.
As in \cite{FullyNonlinear}, we say $u\in C^0(\Omega)$ is a viscosity subsolution (supersolution) of \eqref{equ:SPL} in $\Omega\subset\mathbb R^n$ if for any function $\psi\in C^2(\Omega)$ with $D^2\psi>0$, $D^2\psi>(-a+b)I$, $D^2\psi>-I$, $D^2\psi>-(a+b)I$ for $\tau=0$, $\tau\in(0,\frac{\pi}{4})$, $\tau=\frac{\pi}{4}$, $\tau\in(\frac{\pi}{4},\frac{\pi}{2})$ respectively and point $\overline x\in\Omega$ satisfying
$$
\psi(\overline x)=u(\overline x)\quad\text{and}\quad\psi\geq (\leq ) u\quad\text{in }\Omega,
$$
we have
$$
F_{\tau}(\lambda(D^2\psi(\overline x)))\geq (\leq ) C_0.
$$
If $u$ is both viscosity subsolution and supersolution of \eqref{equ:SPL}, we say it is a viscosity solution.

By a direct computation (see for instance \cite{huang2019entire,Warren}), if
 $\lambda_i>-a-b,$ for $i=1,2,\cdots,n$, then
  \begin{equation}\label{equ:identity}
  \sum_{i=1}^n\arctan \frac{\lambda_{i}+a-b}{\lambda_{i}+a+b}=\sum_{i=1}^n\arctan \left(\frac{\lambda_{i}+a}{b}\right)-\frac{n\pi}{4}.
  \end{equation}
Hence replacing $u(x)$ by
\begin{equation}\label{equ:translation}
\left\{
\begin{array}{llllll}
u(x)+\frac{(a-b)}{2}|x|^2, & 0<\tau<\frac{\pi}{4},\\
u(x)+\frac{1}{2}|x|^{2}, & \tau=\frac{\pi}{4},\\
\frac{u(x)}{b}+\frac{a}{2 b}|x|^{2}, & \frac{\pi}{4}<\tau<\frac{\pi}{2},\\
\end{array}
\right.
\end{equation}
it is equivalent to consider
\begin{equation}\label{equ:translatedequ}
G_{\tau}(\lambda(D^2u))=C_0\quad\text{in }\mathbb R^n\setminus\{0\},
\end{equation}
where
$$
G_{\tau}(\lambda):=\left\{
\begin{array}{llllllll}
\displaystyle  \frac{1}{n} \sum_{i=1}^{n} \ln \lambda_{i}, & \tau=0,\\
\displaystyle \frac{\sqrt{a^{2}+1}}{2 b} \sum_{i=1}^{n} \ln \frac{\lambda_{i}}{\lambda_{i}+2b},
  & 0<\tau<\frac{\pi}{4},\\
\displaystyle -\sqrt{2} \sum_{i=1}^{n} \frac{1}{\lambda_{i}}, & \tau=\frac{\pi}{4},\\
\displaystyle \frac{\sqrt{a^{2}+1}}{b} \left(\sum_{i=1}^{n} \arctan \lambda_i-\frac{n\pi}{4}\right) &
  \frac{\pi}{4}<\tau<\frac{\pi}{2},\\
\displaystyle \sum_{i=1}^{n} \arctan \lambda_{i}, & \tau=\frac{\pi}{2}.\\
\end{array}
\right.
$$

Our first main result provides a type of symmetry and relationship between asymptotic at infinity and value at origin.
\begin{theorem}\label{Thm:classification}
  Let $u\in  C^2(\mathbb R^n\setminus\{0\})$ be a classical convex solution of \eqref{equ:translatedequ}.
  Then
  $u\in C^0(\mathbb R^n)$,
  \begin{equation}\label{equ:Entire}
  G_{\tau}(\lambda(D^2u))\geq C_0\quad\text{in }\mathbb R^n
  \end{equation}
  in viscosity sense
  and
  there exist $c\in\mathbb R$, $\beta\in\mathbb R^n$ and $A\in\mathtt{Sym}(n)$ with $G_{\tau}(\lambda(A))=C_0$  such that
  \begin{equation}\label{equ:c-Geq-U0}
  u(x)\leq \frac{1}{2}x^TAx+\beta  x+c\quad\text{in }\mathbb R^n.
  \end{equation}
  Furthermore, $u(x)-\beta x$ is symmetric in eigenvector directions of $A$, i.e.,
  \begin{equation}\label{equ:symmetry}
  u(\widetilde x)-\beta \widetilde x=u(x)-\beta x,\quad\forall~\{\widetilde x\in\mathbb R^n:(O^T\widetilde x)_i=\pm(O^T x)_i,\quad\forall~i=1,2,\cdots,n\},
  \end{equation}
  where $O$ is an orthogonal matrix such that $OAO^T$ is diagonal.
\end{theorem}

\iffalse

with either of the following holds  in $\mathbb R^n\setminus\{0\}$
  \begin{enumerate}[(i)]
  \item \label{case-MA} $D^2u>0$  for $\tau=0$;
  \item \label{case-small} $D^2u>0$  for $\tau\in(0,\frac{\pi}{4})$;
  \item \label{case-mid} $D^2u>0$   for $\tau=\frac{\pi}{4}$;
  \item \label{case-large} either
\begin{equation}\label{condition-temp-2}
    D^2u>
    \left\{
    \begin{array}{lll}
      -I, & n\leq 4,\\
      -\min\left\{1,\frac{1}{\sqrt 3}+\epsilon(n)\right\}I, & n\geq 5,\\
    \end{array}
    \right.
    \end{equation}
    where $\epsilon(n)$ is as in \eqref{equ-cond-spl}
     or
   \begin{equation}\label{condition-temp-3}
    D^2u> -I\quad\text{and}\quad \frac{bC_0}{\sqrt{a^2+1}}+\frac{n\pi}{4} >\frac{n-2}{2}\pi
    \end{equation}
for $\tau\in(\frac{\pi}{4},\frac{\pi}{2})$;
\item \label{case-SPL} \eqref{equ-cond-spl} for $\tau=\frac{\pi}{2}$.
  \end{enumerate}

\fi

\begin{remark}
  After a change of coordinate $y:=O^Tx$, $u(y)-\beta y$ is symmetric with respect to $n$ coordinate planes. Especially, $u(y)-\beta y$ is central symmetry i.e., $u(y)-\beta y=u(-y)+\beta y$.
\end{remark}
\begin{remark}
  The solutions of \eqref{equ:SPL} and \eqref{equ:translatedequ} have asymptotic behavior and higher order expansion at infinity, see for instance earlier results by the authors \cite{bao-liu2020asymptotic,bao-liu2021asymptotic}.
\end{remark}
\begin{remark}\label{Rem:nonConvex}
  For $\tau \in\left(\frac{\pi}{4}, \frac{\pi}{2}\right)$ case, the convex assumption on $u$ in Theorem \ref{Thm:classification}  can be relaxed to either of the following
\begin{equation}\label{condition-temp-2}
    \left\{
    \begin{array}{lll}
     D^2u> -I, & n\leq 4,\\
     D^2u> -\min\left\{1,\frac{1}{\sqrt 3}+\epsilon(n)\right\}I, & n\geq 5,\\
     D^2u> -I\quad\text{and}\quad\frac{bC_0}{\sqrt{a^2+1}}+\frac{n\pi}{4} >\frac{n-2}{2}\pi, & n\geq 5,\\
    \end{array}
    \right.
    \end{equation}
    where $\epsilon(n)$ is as in \eqref{equ-cond-spl}.
For  $\tau=\frac{\pi}{2}$ case,  the convex assumption  can be relaxed to \eqref{equ-cond-spl}.
\end{remark}

Our second main result shows that the symmetry \eqref{equ:symmetry} is  optimal
for $\tau\in(0,\frac{\pi}{2}]$ cases
 in the following sense.
\begin{proposition}\label{prop:no-genearlizedSol} Let $0<A\in\mathtt{Sym}(n)$ satisfy $G_{\tau}(\lambda(A))=C_0$. If there exists a generalized symmetric convex classical solution $u(x)=U(\frac{1}{2}x^TAx)$ of \eqref{equ:translatedequ}.
 Then for $\tau=0$ case,
 $$
 u=e^{C_0}\int_0^{(x^TAx)^{\frac{1}{2}}}(r^n+c_1)^{\frac{1}{n}}dr+c_2
 $$
 for some $c_1\geq 0$ and $c_2\in\mathbb R$.
 For cases $\tau\in(0,\frac{\pi}{4})$, $\tau=\frac{\pi}{4}$,
  $\tau\in(\frac{\pi}{4},\frac{\pi}{2})$ and $\tau=\frac{\pi}{2}$,  either $A$ is a multiplication of identity matrix i.e.,
  $$
  A=\left(
  \frac{2b}{1-\exp(\frac{2b}{n\sqrt{a^2+1}}C_0)}-2b
  \right)I,\quad A=- \frac{\sqrt 2C_0}{2n} I,
  $$
  $$
  A= \tan\left(\frac{b}{n\sqrt{a^2+1}}C_0+\frac{\pi}{4}\right)I,\quad
  A=\tan\left(\frac{C_0}{n}\right)I
  $$
 respectively,
  or $u$ is the quadratic function $\frac{1}{2}x^TAx$ up to some constant $c$.
\end{proposition}
\begin{remark}
  Proposition \ref{prop:no-genearlizedSol} still holds if equation \eqref{equ:translatedequ} is restricted to an annulus $B_{R_2}\setminus\overline{B_{R_1}}$ with any $0<R_1<R_2$, where $B_r$ denotes the ball centered at origin with radius $r$.
\end{remark}

By \eqref{equ:c-Geq-U0} in Theorem \ref{Thm:classification}, $c\geq u(0)$.
Our final main result shows that $c\geq u(0)$ is also a sufficient condition for a classical solution exists. More explicitly, we prove
the existence of solution of \eqref{equ:SPL}, $\tau\in (0,\frac{\pi}{4})$ with prescribed value at origin and asymptotic behavior at infinity  for all $c\geq u(0)$.

\begin{theorem}\label{Thm:existence}
  For any given $ c\geq u_0, \beta\in \mathbb R^n$ and $0<A\in\mathtt{Sym}(n)$ with $G_{\tau}(\lambda(A))=C_0$ satisfying
  \begin{equation}\label{equ:A-condition-original}
  \delta_0:=
 \sum_{i=1}^n\frac{\min\{\lambda_1(A),\cdots,\lambda_n(A)\}+2b}{\lambda_i(A)+2b}>2,
  \end{equation}
  there exists a unique convex viscosity solution $u\in C^0(\mathbb R^n)$  of
  \begin{equation}\label{equ:Sys-Dirichlet}
  \left\{
  \begin{array}{llllll}
    G_{\tau}(\lambda(D^2u))=C_0, & \text{in }\mathbb R^n\setminus\{0\},\\
    u(0)=u_0, \\
    u(x)= \frac{1}{2}x^TAx+\beta x+c+O(|x|^{2-\delta_0}), & \text{as }|x|\rightarrow\infty,\\
  \end{array}
  \right.
  \end{equation}
  with $\tau\in (0,\frac{\pi}{4})$.
\end{theorem}
\begin{remark} At the end of Section \ref{Sec:Existence}, we prove that the solution found in Theorem \ref{Thm:existence} also satisfies \eqref{equ:Entire} and \eqref{equ:c-Geq-U0}.
  Condition \eqref{equ:A-condition-original} holds if $A$ has two same minimal eigenvalues. For any $C_0<0$, condition \eqref{equ:A-condition-original} is possible to fail. For instance there exists sequences such that $G_{\tau}(\lambda_1,\lambda_2,\lambda_3)=C_0$ but
  $$
  \lambda_1(\epsilon)\rightarrow \frac{2b}{1-\exp(\frac{2b}{\sqrt{a^2+1}}C_0)}-2b\quad\text{and}\quad
  \lambda_2(\epsilon),\lambda_3(\epsilon)\rightarrow+\infty
  $$
  as $\epsilon\rightarrow 0^+$. By a direct computation, the left hand side of \eqref{equ:A-condition-original} tends to $1$ as $\epsilon\rightarrow 0^+$.
  Whether \eqref{equ:A-condition-original} is optimal remains unknown.
\end{remark}
\begin{remark}
 The study on exterior Dirichlet problem by Li-Li \cite{Li-Li-Dirichlet-HessianQuotient} for $\tau=\frac{\pi}{4}$ and Li \cite{Li-Dirichlet-SPL} for $\tau\in(\frac{\pi}{4},\frac{\pi}{2}]$  also works for punctured space after minor modification. From their proof, the results in Theorem \ref{Thm:existence} still holds with  $\delta_0$ changed into
 $$
 \delta_0=\dfrac{\sigma_{n-1}(\lambda(A))\cdot \min\{\lambda_1(A),\cdots,\lambda_n(A)\}}{\sigma_{n}(\lambda(A))}
 $$
 for $\tau=\frac{\pi}{4}$,
 $$
 \delta_0=\frac{\sum_{k=0}^{n} k c_{k}(\frac{n\pi}{4}+\frac{b}{\sqrt{a^2+1}}C_0) \sigma_{k}(\lambda(A))}{\sum_{k=0}^{n} \xi_{k}(\frac{n\pi}{4}+\frac{b}{\sqrt{a^2+1}}C_0, \lambda(A)) c_{k}(\frac{n\pi}{4}+\frac{b}{\sqrt{a^2+1}}C_0) \sigma_{k}(\lambda(A))}
 $$
 for $\tau\in(\frac{\pi}{4},\frac{\pi}{2})$ and
 $$
 \delta_0=\frac{\sum_{k=0}^{n} k c_{k}(C_0) \sigma_{k}(\lambda(A))}{\sum_{k=0}^{n} \xi_{k}(C_0, \lambda(A)) c_{k}(C_0) \sigma_{k}(\lambda(A))}
 $$
 for $\tau=\frac{\pi}{2}$, where
 \begin{equation*}
c_{k}(C):=\left\{\begin{array}{ll}
c_{2 j}(C):=(-1)^{j+1} \sin C & \text { if } k=2 j,~j\in\mathbb N,\\
c_{2 j+1}(C):=(-1)^{j} \cos C & \text { if } k=2 j+1,~j\in\mathbb N,\\
\end{array}\right.
\end{equation*}
\begin{equation*}
\Xi_{k}(\lambda(A), x):=\frac{\sum_{i=1}^{n} \sigma_{k-1 ; i}(\lambda(A)) \lambda_{i}^{2}(A) x_{i}^{2}}{\sigma_{k}(\lambda(A)) \sum_{i=1}^{n} \lambda_{i}(A) x_{i}^{2}}\quad \forall~ x \in \mathbb{R}^{n} \backslash\{0\},
\end{equation*}
\begin{equation*}
\bar{\xi}_{k}(\lambda(A)):=\sup _{x \in \mathbb{R}^{n} \backslash\{0\}}
\Xi_{k}(\lambda(A), x),\quad
\underline{\xi}_{k}(\lambda(A)):=\inf _{x \in \mathbb{R}^{n} \backslash\{0\}} \Xi_{k}(\lambda(A), x),
\end{equation*}
and
\begin{equation*}
\xi_{k}(C, \lambda(A)):=\left\{\begin{array}{ll}
\bar{\xi}_{k}(\lambda(A)) & \text { if } c_{k}(C)>0, \\
\underline{\xi}_{k}(\lambda(A)) & \text { if } c_{k}(C) \leq 0.\\
\end{array}\right.
\end{equation*}
  For $\tau=0$, optimal existence of Dirichlet problem on exterior domain and punctured space have been studied in \cite{Bao-Li-Zhang-CVPDE2015,CL,Jin-Xiong,Li-Lu-2018} etc.
\end{remark}

The paper is organized as follows. In the next three sections, we prove   Theorem \ref{Thm:classification}, Theorem \ref{Thm:existence} and Proposition \ref{prop:no-genearlizedSol} respectively.

\section{Proof of Theorem \ref{Thm:classification}}\label{Sec:Symmetry}

In this section, we prove Theorem \ref{Thm:classification} by maximum principle and removable singularity of viscosity solutions. The following knowledge on convex functions is necessary.
\begin{theorem}[Proposition 2.1 of \cite{Jin-Xiong}]\label{thm:2.1}
  Let $v$ be a locally convex function in $B_1\setminus\{0\}$ with $n\geq 2$. Then $u$ can be uniquely extended to be a convex function in $B_1$.
\end{theorem}
\begin{theorem}[Theorem 1.1 of \cite{CLN-Remarks-III}]\label{thm:2.2}
  Let   $\Omega\subset\mathbb R^n$ be a domain, $\overline x\in\Omega$ and  $u\in C^0(\Omega)$ be a  viscosity solution of
  $$
  F(D^2u)\geq f(x)\quad\text{in }\Omega\setminus\{\overline x\},
  $$
  where $F\in C^0(\mathtt{Sym}(n))$ is an elliptic operator and $f\in C^0(\Omega)$. Suppose that $u$ is upperconical at $\overline x$, i.e.,
  \begin{equation}\label{equ:lowerconical}
    \sup_{x\in\Omega}\left(
    (u+\eta)(x)-(u+\eta)(\overline x)
    +\epsilon|x-\overline x|
    \right)>0
  \end{equation}
  for any $\eta\in C^{\infty}(\Omega)$ and $\epsilon>0$. Then
  $$
  F(D^2u)\geq f(x)\quad\text{in }\Omega
  $$
  in viscosity sense.
\end{theorem}
\begin{remark}
  If $u$ is a convex function in $\Omega$, then for all $\overline x\in\Omega$, $u$ is upperconical at $\overline x$. By contradiction, we suppose there exist $\epsilon\in (0,1)$ and $\eta\in C^{\infty}(\Omega)$ such that
  $$
  (u+\eta)(x)-(u+\eta)(\overline x)+\epsilon|x-\overline x|\leq 0\quad\text{in }\Omega.
  $$
  By Taylor expansion of $\eta$ near $\overline x$, there exists $0<\delta<1$ such that
  $$
  u(x)\leq u(\overline x)-\nabla\eta(\overline x)\cdot (x-\overline x)-\frac{\epsilon}{2}|x-\overline x|,\quad\forall~ 0<|x-\overline x|<\delta.
  $$
  Since $u$ is convex, take any $e\in\partial B_1$ and we have
  $$
  2u(\overline x)\leq u(\overline x+\frac{\delta}{2}e)+
  u(\overline x-\frac{\delta}{2}e)\leq 2u(\overline x)-\frac{\epsilon\delta}{4}<2u(\overline x),
  $$
  which is a contradiction.

  For more examples and discussions on
  condition \eqref{equ:lowerconical}, we refer to Remark 1.2 of \cite{CLN-Remarks-III}.
\end{remark}

Let $u\in  C^2(\mathbb R^n\setminus\{0\})$ be a classical  solution of \eqref{equ:translatedequ} as in Theorem \ref{Thm:classification} or Remark \ref{Rem:nonConvex}. By the asymptotic behavior result (see for instance Theorem 1.3 of \cite{bao-liu2020asymptotic}), there exists $c\in\mathbb R, \beta\in\mathbb R^n$ and  $A\in\mathtt{Sym}(n)$ with $G_{\tau}(\lambda(A))=C_0$ such that
  \begin{equation}\label{equ:asymptotics}
 \limsup_{|x|\rightarrow+\infty}|x|^{2+k-n}\left|D^k \left(u(x)-\left(
  \frac{1}{2}x^TAx+\beta  x+c
  \right)\right)\right|<\infty
  \end{equation}
for all $k\in\mathbb N$.

Step 1. we prove \eqref{equ:Entire} and \eqref{equ:c-Geq-U0}.
The proof is separated into five cases according to $\tau$.

If $u$ is a classical convex solution as in Theorem \ref{Thm:classification},  \eqref{equ:Entire} follows immediately from Theorems \ref{thm:2.1} and \ref{thm:2.2}. By \eqref{equ:asymptotics}, we have
$$
u(x)\rightarrow  \frac{1}{2}x^TAx+\beta x+c =:V(x)\quad\text{as }|x|\rightarrow\infty.
$$
By a direct computation, $V(x)$ is a quadratic solution of $G_{\tau}(\lambda(D^2u))=C_0$ on entire $\mathbb R^n$. By maximum principle  (see for instance Theorem 17.1 of \cite{Book-GT}),
$$
u(x)\leq V(x)=\frac{1}{2}x^TAx+\beta x+c\quad\text{in }\mathbb R^n.
$$

For $\tau\in(\frac{\pi}{4},\frac{\pi}{2})$ case in Remark \ref{Rem:nonConvex}, the desired result follows from $\tau=\frac{\pi}{2}$ case, which will be proved immediately. For $\tau=\frac{\pi}{2}$ case of Remark \ref{Rem:nonConvex}, we prove that there exists $K_0\geq 0$ such that
$$
D^2u>-K_0I.
$$
For the first subcase of \eqref{equ-cond-spl}, the result is immediate.
For $C_0>\frac{n-2}{2}\pi$, then for all $i=1,2,\cdots,n$,
$$
\arctan\lambda_i(D^2u)+\frac{n-1}{2}\pi>\sum_{j=1}^n\arctan\lambda_j(D^2u)=C_0.
$$
Thus we may take $K_0:=\min\{\tan(C_0-\frac{n-1}{2}\pi),0\}$.
Let $$
v(x):=u(x)+\frac{K_0}{2}|x|^2,
$$
then
$$
D^2v=D^2u+K_0I>0\quad\text{in }\mathbb R^n\setminus\{0\}
$$
and $v$ satisfies
$$
\sum_{i=1}^n\arctan(\lambda_i(D^2v)-K_0)=C_0\quad\text{in }\mathbb R^n\setminus\{0\}.
$$
By a direct computation,
$$
\frac{\partial
}{\partial\lambda_i}
 \arctan(\lambda_i-K_0)=\frac{1}{1+(\lambda_i-K_0)^2}>0.
$$
Together with asymptotics \eqref{equ:asymptotics}, $D^2v$ is bounded from above. Hence the differential operator is elliptic  (see for instance \cite{CNS-DirichletIII}) and the desired results follow from Theorems \ref{thm:2.1} and \ref{thm:2.2}.

Step 2.  we prove the symmetric property \eqref{equ:symmetry} for diagonal case  i.e., $A=\mathtt{diag}(a_1,a_2,\cdots,a_n)$. We only need to prove for $\beta=0$ case, otherwise replace $u$ by $u-\beta x$.  Similar to the proof of Theorem 1.1 in \cite{Jin-Xiong}, we introduce the following maximum principle.
\begin{lemma}\label{Lem:maximumPrinciple}
  Let $w$ be a classical solution of
  $$
  a_{ij}D_{ij}w=0\quad\text{in }\mathbb R^n\setminus\{0\},
  $$
  where $a_{ij}(x)$ is a positive definite matrix. Suppose
  $$
  w(0)=0\quad\text{and}\quad\lim_{|x|\rightarrow\infty}w(x)=0,
  $$
  then $w\equiv 0$.
\end{lemma}
\begin{proof}
  By contradiction, we may suppose there exists $x_0\in\mathbb R^n\setminus\{0\}$ such that $u(x_0)>0$. Then by continuity and boundary conditions, there exists $x_1\in\mathbb R^n\setminus\{0\} $ such that $u(x_1)=\max_{\mathbb R^n\setminus\{0\}}u>0$. Applying standard maximum principle in level set $\{x\in\mathbb R^n\setminus \{0\}:u>\frac{1}{2}u(x_1)\}$, we have a contradiction $\frac{1}{2}u(x_1)=u(x_1)$.
\end{proof}
For any $i=1,2,\cdots,n$, let
$$
U(x):=u(\widetilde x),\quad\text{where}\quad \widetilde x=(x_1,\cdots,x_{i-1},-x_i,x_{i+1},\cdots,x_n).
$$
By a direct computation,   $u,U$ satisfy
$$
\left\{
\begin{array}{llll}
  G_{\tau}(\lambda(D^2u))=G_{\tau}(\lambda(D^2U))=C_0,& \text{in }\mathbb R^n\setminus\{0\},\\
  u(x)=U(x), & \text{at }x=0,\\
  u(x),U(x)\rightarrow\left( \frac{1}{2}x^TAx+c\right),& \text{as }|x|\rightarrow\infty.
\end{array}
\right.
$$
By Newton-Leibnitz formula between $u,U$ and applying maximum principle as in Lemma \ref{Lem:maximumPrinciple}, we have $u(x)=U(x)=u(\widetilde x)$. Hence $u(x)$ is hyperplane symmetry with respect to all $n$ coordinate hyperplane.

Step 3. we prove \eqref{equ:symmetry} for general $A$.  Since $A$ is a symmetric matrix, by eigen-decomposition there exists an orthogonal matrix $O$ such that
\begin{equation}\label{equ:eigendecomposition-2}
A=O^T\Lambda O,\quad\text{where}\quad \Lambda=\mathtt{diag}(a_1,\cdots,a_n).
\end{equation}
Let $v(x):=u(O^{-1}x)$, then $D^2v(x)=OD^2u(O^{-1}x)O^{-1}$ has the same eigenvalues as $D^2u(O^{-1}x)$. Hence
$$
G_{\tau}(\lambda(D^2v))=C_0\quad\text{in }\mathbb R^n\setminus\{0\}
$$
and
$$
v(0)=u(0),\quad v(x)\rightarrow \frac{1}{2}x^TOAO^{-1}x+c=\frac{1}{2}x^T\Lambda x+c
$$
as $|x|\rightarrow\infty$. By the previous step, we have
$$
v(\widetilde x)=v(x)\quad\forall~ \{\widetilde x\in\mathbb R^n~:~
\widetilde x_i=\pm x_i,\quad\forall~i=1,2,\cdots,n
\}.
$$
Thus
$$
u(\widetilde x)=u(x)\quad\forall~\{\widetilde x\in\mathbb R^n~:~ (O^T\widetilde x)_i=\pm (O^Tx)_i,\quad\forall~i=1,2,\cdots,n\}
$$
and \eqref{equ:symmetry} follows immediately.

\section{Proof of Theorem \ref{Thm:existence}}\label{Sec:Existence}

In this section, we consider the existence result of Dirichlet problem for
$\tau\in(0,\frac{\pi}{4})$ with prescribed asymptotic behavior and value at origin. By similar translation as \eqref{equ:translation} and eigen-decomposition as \eqref{equ:eigendecomposition-2}, we may assume $\beta=0$ and consider only  the following Dirichlet problem.
\begin{equation}\label{equ:Translated}
\left\{
\begin{array}{lllll}
  G_{\tau}(\lambda(D^2u))=C_0,& \text{in }\mathbb R^n\setminus\{0\},\\
  u(0)=u_0,\\
  u(x)\rightarrow\left(\frac{1}{2}x^TAx+c\right),& \text{as }|x|\rightarrow\infty,\\
\end{array}
\right.
\end{equation}
where
$$
0<A=\mathtt{diag}(a_1,\cdots,a_n)\quad\text{satisfies}\quad G_{\tau}(a_1,\cdots,a_n)= C_0<0.
$$
We will write the vector $(a_1,\cdots,a_n)$ simply as $a$ when no confusion can arise.

Similar to the strategy as in \cite{Bao-Li-Li-2014,Li-Li-Dirichlet-HessianQuotient,Li-Dirichlet-SPL},
we seek for generalized symmetric subsolution of form
\begin{equation}\label{equ:def-s}
u(x)=U(s),\quad\text{where}\quad
s(x):=\frac{1}{2}\sum_{i=1}^na_ix_i^2,
\end{equation} and apply Perron's method. By a direct computation,
$$
D_iu(x)=U'(s)a_ix_i,\quad D_{ij}u(x)=U'(s)a_i\delta_{ij}+U''(s)a_ia_jx_ix_j.
$$
By Sylvester's determinant theorem, we have
$$
\begin{array}{lllll}
&\sigma_n(\lambda(D^2u)) \\
=& \displaystyle \det \left(
\left(\begin{array}{lll}
  U'a_1\\
  & \ddots\\
  &&U'a_n\\
\end{array}\right)
+U''(s)\left(
\begin{array}{c}
  a_1x_1\\
  \vdots\\
  a_nx_n\\
\end{array}
\right)\cdot (a_1x_1,\cdots,a_nx_n)
\right)\\
=&\displaystyle
\det \left(\begin{array}{lll}
  U'a_1\\
  & \ddots\\
  &&U'a_n\\
\end{array}\right)\\
&\displaystyle \cdot \left(
1+U'' (a_1x_1,\cdots,a_nx_n)
\left(\begin{array}{lll}
  (U'a_1)^{-1}\\
  & \ddots\\
  &&(U'a_n)^{-1}\\
\end{array}\right)
\left(
\begin{array}{c}
  a_1x_1\\
  \vdots\\
  a_nx_n\\
\end{array}
\right)
\right)\\
=&\displaystyle (U')^n\sigma_n(a)+U''(U')^{n-1}\sigma_n(a)\sum_{i=1}^n
a_ix_i^2,
\end{array}
$$
i.e.,
\begin{equation}\label{equ:subsolu-1}
\sigma_n(\lambda(D^2u))=(U')^n\sigma_n(a)+2sU''(U')^{n-1}\sigma_n(a).
\end{equation}
Similarly,
$$
D_{ij}(u+b|x|^2)=(U'a_i+2b)\delta_{ij}+U''a_ia_jx_ix_j
$$
and
$$
\begin{array}{lllll}
&\sigma_n(\lambda(D^2(u+b|x|^2))) \\
=& \displaystyle \det \left(
\left(\begin{array}{lll}
  U'a_1+2b\\
  & \ddots\\
  &&U'a_n+2b\\
\end{array}\right)
+U''(s)\left(
\begin{array}{c}
  a_1x_1\\
  \vdots\\
  a_nx_n\\
\end{array}
\right)\cdot (a_1x_1,\cdots,a_nx_n)
\right)\\
=&\displaystyle
\det \left(\begin{array}{lll}
  U'a_1+2b\\
  & \ddots\\
  &&U'a_n+2b\\
\end{array}\right)\\
&\displaystyle\cdot \left(
1+U'' (a_1x_1,\cdots,a_nx_n)
\left(\begin{array}{lll}
  (U'a_1+2b)^{-1}\\
  & \ddots\\
  &&(U'a_n+2b)^{-1}\\
\end{array}\right)
\left(
\begin{array}{c}
  a_1x_1\\
  \vdots\\
  a_nx_n\\
\end{array}
\right)
\right)\\
=&\displaystyle
\prod_{i=1}^n(U'a_i+2b)
+U''\sum_{i=1}^n a_i^2x_i^2\prod_{j=1,\cdots,n\atop j\not=i}(U'a_j+2b)
\end{array}
$$
i.e.,
\begin{equation}\label{equ:subsolu-2}
\sigma_n(\lambda(D^2(u+b|x|^2)))=\sigma_n(a)\prod_{i=1}^n(U'+\frac{2b}{a_i})
+U''\sigma_n(a)\sum_{i=1}^na_ix_i^2\prod_{j=1,\cdots,n\atop j\not=i}(U'+\frac{2b}{a_j}).
\end{equation}
By \eqref{equ:subsolu-1} and \eqref{equ:subsolu-2}, $u(x)=U(s)$ is a subsolution of \eqref{equ:translatedequ} as long as
$$
\dfrac{\displaystyle(U')^n\sigma_n(a)
+2sU''(U')^{n-1}\sigma_n(a)}{
\displaystyle \prod_{i=1}^n(U'a_i+2b)
+U''\sigma_n(a)\sum_{i=1}^n\left(a_ix_i^2\prod_{j=1,\cdots,n\atop j\not=i}(U'+\frac{2b}{a_j})\right)
}\geq c_0,
$$
i.e.,
\begin{equation}\label{equ:denominator}
\dfrac{\displaystyle(U')^n
+2sU''(U')^{n-1}}{
\displaystyle \prod_{i=1}^n(U'+\frac{2b}{a_i})
+U''\sum_{i=1}^n\left(a_ix_i^2\prod_{j=1,\cdots,n\atop j\not=i}(U'+\frac{2b}{a_j})\right)
}\geq c_0,
\end{equation}
where $c_0=\exp(\frac{2b}{\sqrt{a^2+1}}C_0)\in (0,1)$ satisfies
\begin{equation}\label{equ:def-c_0}
c_0=\prod_{i=1}^n\frac{a_i}{a_i+2b},\quad\text{i.e.,}\quad
1=c_0\cdot\prod_{i=1}^n(1+\frac{2b}{a_i}).
\end{equation}
Hereinafter, we may assume without loss of generality that
$$
0<a_1\leq a_2\leq\cdots\leq a_n.
$$
\begin{lemma}\label{Lem:subsolution-0}
Let $U$ be a classical solution of
  \begin{equation}\label{equ:ODE}
    \left\{
    \begin{array}{llllll}
      \displaystyle (U')^n-c_0\prod_{i=1}^n(U'+\frac{2b}{a_i})\\
      \displaystyle +2(s+1)U''\left((U')^{n-1}
      -c_0\prod_{i=2}^n(U'+\frac{2b}{a_i})
      \right)=0,&\text{in }s>0,\\
      U'\geq 1,~U''\leq 0,& \text{in }s>0.
    \end{array}
    \right.
  \end{equation}
Then $u(x)=U(s)$ is convex and satisfies $G_{\tau}(\lambda(D^2u))\geq  C_0$ in $\mathbb R^n\setminus\{0\}$.
\end{lemma}
\begin{proof}
By the assumptions on $a_i$, for all $U'\geq 1$,
$$
\sum_{i=1}^n\left(a_ix_i^2\prod_{j=1,\cdots,n\atop j\not=i}(U'+\frac{2b}{a_j})\right)\leq
\sum_{i=1}^n\left(a_ix_i^2\prod_{j=1}^{n-1}(U'+\frac{2b}{a_j})\right)=2s\prod_{j=1}^{n-1}(U'+\frac{2b}{a_j}).
$$
We claim that the denominator part of \eqref{equ:denominator} is positive. By $U'\geq 1,~U''\leq 0$, we only need to prove
\begin{equation}\label{equ:temp-3}
U'+2sU''> 0,
\end{equation}
then
$$
\begin{array}{llllll}
& \displaystyle \prod_{i=1}^n(U'+\frac{2b}{a_i})
+U''\sum_{i=1}^n\left(a_ix_i^2\prod_{j=1,\cdots,n\atop j\not=i}(U'+\frac{2b}{a_j})\right)\\
\geq &\displaystyle  \prod_{i=1}^n(U'+\frac{2b}{a_i})
+2sU'' \prod_{i=1}^{n-1}(U'+\frac{2b}{a_i})\\
= & \displaystyle
\left(\prod_{i=1}^{n-1}(U'+\frac{2b}{a_i})\right)\cdot \left(
U'+\frac{2b}{a_n}+2sU''
\right)\\
>&0.
\end{array}
$$
By \eqref{equ:ODE}, we have
$$
2sU''\geq 2(s+1)U''=-
\dfrac{\displaystyle (U')^n-c_0\prod_{i=1}^n(U'+\frac{2b}{a_i})}{
\displaystyle (U')^{n-1}-c_0\prod_{i=2}^n(U'+\frac{2b}{a_i})
}.
$$
Consequently,
$$
U'+ 2sU''\geq
\dfrac{\displaystyle  c_0\left(
\prod_{i=1}^n(U'+\frac{2b}{a_i})-U'\prod_{i=2}^n(U'+\frac{2b}{a_i})\right)}{
\displaystyle (U')^{n-1}-c_0\prod_{i=2}^n(U'+\frac{2b}{a_i})
}.
$$
Notice that for all $U'\geq 1$, in view of \eqref{equ:def-c_0},
$$
(U')^{n-1}-c_0\prod_{i=2}^n(U'+\frac{2b}{a_i})
=(U')^{n-1}\left(
1-c_0\prod_{i=2 }^n(1+\frac{2b}{a_i U'})
\right)>0.
$$
Combining the calculus above,   \eqref{equ:temp-3} is proved.

By \eqref{equ:subsolu-1} and Proposition 1.2 of \cite{Bao-Li-Li-2014},
\begin{equation}\label{equ:represent-sigmak}
  \sigma_k(\lambda(D^2u))=(U')^k\sigma_k(a)+U''(U')^{k-1}
  \sum_{i=1}^n(a_ix_i)^2\sigma_{k-1;i}(a).
\end{equation}
Since $U'\geq 1$ and $U''\leq 0$, together with \eqref{equ:temp-3} and the fact that
$$
\sigma_k(a)=\sigma_{k;i}(a)+a_i\sigma_{k-1;i}(a),
$$
for all $k=1,2,\cdots,n$,
we have
$$
\begin{array}{llll}
   \sigma_{k}\left(\lambda\left(D^{2} u\right)\right) &= &
   \displaystyle \left(U^{\prime}\right)^{k} \sigma_{k}(a)+U^{\prime \prime}\left(U^{\prime}\right)^{k-1} \sum_{i=1}^{n}a_ix_i^2 (\sigma_k(a)-\sigma_{k;i}(a))
   \\
&\geq & \left(U^{\prime}\right)^{k} \sigma_{k}(a)+2sU''(U')^{k-1}\sigma_k(a)>0.\\
\end{array}
$$
It follows that $D^2u>0$ in $\mathbb R^n\setminus\{0\}$.

By the calculus above, \eqref{equ:denominator} holds as long as
$$
\begin{array}{lll}
\displaystyle (U')^n\sigma_n(a)-c_0\prod_{i=1}^n(U'a_i+2b)+2s\sigma_n(a) U''(U')^{n-1} \\
\displaystyle -c_0 U''\sigma_n(a)\sum_{i=1}^n\left(a_ix_i^2
\prod_{j=1,\cdots,n\atop j\not=i}(U'+\frac{2b}{a_j})\right)&
\geq 0.
\end{array}
$$
Dividing both sides by $\sigma_n(a)>0$, it is equivalent to
\begin{equation}\label{equ:target}
  (U')^n-c_0\prod_{i=1}^n(U'+\frac{2b}{a_i})+ 2s U''(U')^{n-1}
-c_0 U''\sum_{i=1}^n\left(a_ix_i^2\prod_{j=1,\cdots,n\atop j\not=i}(U'+\frac{2b}{a_j})\right)\geq 0.
\end{equation}
By the assumptions, we have
$$
\sum_{i=1}^n\left(a_ix_i^2\prod_{j=1,\cdots,n\atop j\not=i}(U'+\frac{2b}{a_j})\right)
\geq 2s
\prod_{i=2}^n(U'+\frac{2b}{a_i})
$$
and
$$
\begin{array}{llll}
&\displaystyle 2s  U''(U')^{n-1}
-c_0 U''\sum_{i=1}^n\left(a_ix_i^2\prod_{j=1,\cdots,n\atop j\not=i}(U'+\frac{2b}{a_j})\right) \\
\geq &\displaystyle 2sU''\left(
(U')^{n-1}
-c_0\prod_{i=2}^n(U'+\frac{2b}{a_i})
\right).
\end{array}
$$
By a direct computation, as long as $U'\geq 1$,
$$
(U')^{n-1}
-c_0\prod_{i=2}^n(U'+\frac{2b}{a_i})
\geq (U')^{n-1}\cdot\left(1-\prod_{i=1}^n(\frac{a_i}{a_i+2b})\prod_{i=2}^n(\frac{a_i+2b}{a_i})\right)
=(U')^{n-1}\cdot\frac{2b}{a_1+2b}>0.
$$
Hence \eqref{equ:target} holds if $U'\geq 1, U''\leq 0$ and
$$
(U')^n-c_0\prod_{i=1}^n(U'+\frac{2b}{a_i})+2(s+1)U''\left((U')^{n-1}
      -c_0\prod_{i=2}^n(U'+\frac{2b}{a_i})
      \right)=0
$$
Especially when $U$ satisfies \eqref{equ:ODE}, $u(x)=U(s)$ becomes a subsolution of  \eqref{equ:Translated} and this finishes the proof.
\end{proof}
\begin{lemma}\label{Lem:Subsolution}
  For any $c\geq 0$, there exists a solution of \eqref{equ:ODE} with
  $U(0)=0$, $U'(0)\geq 1$ and
  \begin{equation}\label{equ:prescribe-asymptotics}
  U(s)=s+c+o(1)\quad\text{as }s\rightarrow\infty,
  \end{equation}
  if $a_1,\cdots,a_n$ satisfies
  \begin{equation}\label{equ:A-condition}
  \sum_{i=1}^n\frac{a_1+2b}{a_i+2b}>2.
  \end{equation}
\end{lemma}
\begin{proof}
  Step 1. existence of initial value problem
  \begin{equation}\label{equ:initialValue}
  \left\{
  \begin{array}{llll}
    \displaystyle \psi'(s)=\frac{g(\psi(s))}{s+1}, & \text{in }s\in(0,\infty),\\
    \psi(0)=\alpha \geq 1,\\
  \end{array}
  \right.
  \end{equation}
  where
  $$
  g(\psi):=-\dfrac{\displaystyle\psi^n-c_0\prod_{i=1}^n(\psi+\frac{2b}{a_i})}{
  \displaystyle 2\left(
  \psi^{n-1}-c_0\prod_{i=2}^n(\psi+\frac{2b}{a_i})
  \right)
  }.
  $$
  By a direct computation and \eqref{equ:def-c_0}, for any $\psi>1$,
  \begin{equation}\label{equ:fenzi}
  \dfrac{c_0\prod_{i=1}^n(\psi+\frac{2b}{a_i})}{\psi^n}=
  c_0\prod_{i=1}^n(1+\frac{2b}{a_i\psi})<c_0\prod_{i=1}^n(1+\frac{2b}{a_i})=1,
  \end{equation}
  for any $\psi\geq 1$,
  \begin{equation}\label{equ:fenmu}
  \dfrac{c_0\prod_{i=2}^n(\psi+\frac{2b}{a_i})}{\psi^{n-1}}=c_0\prod_{i=2}^n(1+\frac{2b}{a_i\psi})
  <c_0\prod_{i=1}^n(1+\frac{2b}{a_i})=1
  \end{equation}
  and $g(1)=1$. Thus for all $\psi>1$,
  $g(\psi)<0$ and $\psi=1$ is the unique root of $g(\psi)=0$ in $\psi\geq 1$.

  If  $\alpha=1$, then $\psi\equiv 1$ is the only solution by the uniqueness theorem of ordinary differential equations. If $\alpha >1$,  by the local existence result of ODE, $\psi$ exists in a neighbourhood of origin.  Since $\psi=1$ is the only root of $g(\psi)=0$ in the range of $\psi\geq 1$ and $g(\psi)<0$ as long as $\psi>1$, $\psi(s)$ is monotone decreasing with respect to $s$. Thus $\psi$ is bounded from above by $\alpha$ and hence by the extension theorem, $\psi$ exists on entire $s\in[0,\infty)$.
  Since $g(1)=0$, by the uniqueness theorem of ODE, $\psi(s)>1$ for all $s\in[0,\infty)$. Hence $\psi(s)$ admits a limit $\psi_{\infty}$ at infinity. We claim that the limit $\psi_{\infty}$ must be $1$, i.e.,
  \begin{equation}\label{equ:temp-4}
  \psi(s)\rightarrow 1\quad\text{as }s\rightarrow+\infty.
  \end{equation}
  By contradiction, if $\psi_{\infty}>1$, then by equation in \eqref{equ:initialValue}, there exists $\epsilon>0$ and $C=C(\psi_{\infty},\epsilon)>0$ such that
  $$
  \psi'(s)=\frac{g(\psi(s))}{s+1}<-\epsilon,\quad\forall~s>C.
  $$
  Integral over $s\in (C,\infty)$ and it contradicts to the fact that $\psi$ converge to $\psi_{\infty}$ at infinity. This finishes the proof of \eqref{equ:temp-4}.

  Step 2. refined asymptotics at infinity. By a direct computation and applying \eqref{equ:fenzi}, \eqref{equ:fenmu},
  $$g'(1)=
  -\frac{1}{2}\left(
  n-c_0\sum_{i=1}^n\prod_{j=1,\cdots,n\atop j\not=i}(1+\frac{2b}{a_j})
  \right)\cdot \left(
  1-c_0\prod_{i=2}^n(1+\frac{2b}{a_i})
  \right)^{-1}.
  $$
  Furthermore by \eqref{equ:def-c_0},
  $$
  c_0\sum_{i=1}^n\prod_{j=1,\cdots,n\atop j\not=i}(1+\frac{2b}{a_j})=
  \sum_{i=1}^n\left(
  \prod_{j=1}^n\frac{a_j}{a_j+2b}\prod_{j=1,\cdots,n\atop j\not=i}(\frac{a_j+2b}{a_j})
  \right)=\sum_{i=1}^n\frac{a_i}{a_i+2b}
  $$
  and
  $$
  \left(
  1-c_0\prod_{i=2}^n(1+\frac{2b}{a_i})
  \right)^{-1}=
  \left(
  1-\prod_{i=1}^n\frac{a_i}{a_i+2b}\prod_{i=2}^n \frac{a_i+2b}{a_i}
  \right)^{-1}=\frac{a_1+2b}{2b}.
  $$
  Thus
  $$
  g'(1)=-\frac{1}{2}\frac{a_1+2b}{2b}\sum_{i=1}^n\frac{2b}{a_i+2b}
  =-\frac{1}{2}\sum_{i=1}^n\frac{a_1+2b}{a_i+2b}=-\frac{\delta_0}{2}
  <-1
  $$
  and $\psi-1$ satisfies
  $$
  (\psi-1)'=\frac{g(\psi)-g(1)}{s+1}\quad\text{and}\quad\psi-1=o(1)
  $$
  as $s\rightarrow\infty$. By standard ODE analysis, we prove
  \begin{equation}\label{equ:asym-infnity-1}
  \psi(s)=1+O(s^{g'(1)})\quad\text{as }s\rightarrow\infty.
  \end{equation}
  In fact,
  $$
  t:=\ln(s+1)\in (0,+\infty),\quad \varphi(t):=\psi(s(t))-1
  $$
  satisfies
  $$
  \varphi'(t)=\psi'(s(t))\cdot e^t=g(\varphi+1)-g(1)=:h(\varphi)\quad\text{in }(0,+\infty).
  $$
  with
  $$
  h(0)=0,\quad h'(0)=g'(1)\quad\text{and}\quad\varphi(t)\rightarrow 0\quad\text{as }t\rightarrow+\infty.
  $$
  Since $g'(1)<0$ and
  $$
  h(\varphi)=g'(1)\varphi+O(\varphi^2)\quad\text{as }\varphi\rightarrow 0,
  $$
  by the asymptotic stability of ODE (see for instance Theorem 1.1 of Chap.13 in \cite{Book-Coddington} or Theorem 2.16 of \cite{Book-AsymptoticIntegration} for more detailed analysis), we have
  $$
  \varphi(t)=O(e^{g'(1)t})\quad\text{as }t\rightarrow+\infty,
  $$
  which is exactly \eqref{equ:asym-infnity-1}.

  Step 3. construction of solution of second order ODE \eqref{equ:ODE}.
  By taking
  $$
  U(s):=u_0+\int_0^s\psi(r)dr,\quad s\geq 0,
  $$
  we have $U'(s)=\psi(s)\geq 1$ and thus $U$ becomes a solution of \eqref{equ:ODE}.

  Step 4. computation on the asymptotics at infinity. By \eqref{equ:asym-infnity-1}, for any $\alpha\geq 1$, there exist a solution of \eqref{equ:initialValue} (denoted by $\psi(s,\alpha)$) and a constant $\mu(\alpha)\geq  0$ such that
  $$
  U(s)=u_0+s+\mu(\alpha)+O(s^{g'(1)+1})\quad\text{as }s\rightarrow\infty.
  $$
  More explicitly, by
  $$
  U(s)-s-u_0=\int_0^s(\psi(r,\alpha)-1)dr,
  $$
  sending $s\rightarrow +\infty$ we have
  \begin{equation}\label{equ:temp-5}
  \mu(\alpha)=\int_0^{+\infty}(\psi(s,\alpha)-1)d s.
  \end{equation}
  By the theorem of the differentiability of the solution with respect to the
initial value, we can differentiate \eqref{equ:initialValue} with respect to $\alpha$ and obtain
$$
v(s):=\frac{\partial\psi(s,\alpha)}{\partial\alpha}\quad\text{satisfies}\quad
\left\{
\begin{array}{llll}
  v'=\frac{g'(\psi(s,\alpha))}{s+1}\cdot v, & \text{in }s\in(0,\infty),\\
  v(0)=1,\\
\end{array}
\right.
$$
which is solved by
$$
v(s)=\exp \left(
\int_0^s\frac{g'(\psi(r,\alpha))}{r+1}d r
\right)>0.
$$
Hence $\psi(s,\alpha)$ is monotone increasing w.r.t. $\alpha$. Now we claim that for any fixed $s>0$,
\begin{equation}\label{equ:temp-6}
\lim_{\alpha\rightarrow+\infty}\psi(s,\alpha)=+\infty.
\end{equation}
By contradiction, we suppose there exist $s_1>0, M>1$ and $\{\alpha_k\}_{k=1}^{\infty}$ such that
$$
1<\alpha_k\rightarrow+\infty,\quad\text{as }k\rightarrow\infty\quad
\text{and}\quad \psi(s_1,\alpha_k)\leq M
$$
for all $k\in\mathbb N.$ By a direct computation, there exists $C>0$ independent of $\beta$ such that
$$
0=g(1)<-g(\psi)<C\psi,\quad\forall~\psi>1.
$$
By \eqref{equ:initialValue}, we have
$$
\dfrac{d\psi}{ C\psi}\leq \frac{d\psi}{-g(\psi)}=-\frac{d s}{s+1}.
$$
Integrate in $\psi\in (M,\alpha_k)$ for large $k$ and using $\psi(1,\alpha_k)=\alpha_k$, we have
$$
\int_M^{\alpha_k}\frac{d\psi}{C\psi}\leq \int_{\psi(s_1,\alpha_k)}^{\psi(1,\alpha_k)}\frac{d\psi}{-g(\psi)}=
-\int_{s_1}^1\frac{ds}{s+1}=\ln \frac{s_1+1}{2}<\infty.
$$
Sending $k\rightarrow+\infty$ and we have
$$
\frac{1}{C}\ln \frac{\alpha_k}{M}<\ln \frac{s_1+1}{2}<+\infty,
$$
which becomes a contradiction. Consequently we prove that
\begin{equation}\label{equ:temp-7}
\mu(\alpha)\rightarrow+\infty\quad\text{as }\alpha\rightarrow+\infty.
\end{equation}
By contradiction, if $\mu(\alpha)$ is uniformly bounded by some $M$ for all $\alpha\geq 1$, then by \eqref{equ:temp-5} we have
$$
2M\geq \int_0^{+\infty}(\psi(s,\alpha)-1)d s.
$$
Fix $s=1$ in \eqref{equ:temp-6}, then there exists sufficiently large $\alpha_0$ such that $\psi(1,\alpha_0)>3M+1$. Since $\psi(s,\alpha)\geq 1$ is monotone decreasing with respect to $s$, we have $\psi(s,\alpha_0)>3M+1$ for all $s\in(0,1)$. This leads to a contradiction and proves \eqref{equ:temp-7}.
On the other hand, when $\alpha=1$, we have
$$
U(s)=u_0+\int_0^s1d r=s+u_0,
$$
which gives $\mu(1)=0$. By the continuity of solution w.r.t $\alpha$, such a solution of \eqref{equ:ODE} with $U(0)=0$ and asymptotics \eqref{equ:prescribe-asymptotics} exists for all $c\geq 0$.
\end{proof}

By standard Perron's method, we have the following existence result.
\begin{theorem}\label{Thm:temp}
  For any given $c\geq u_0$ and $0<A\in\mathtt{Sym}(n)$ satisfying $G_{\tau}(\lambda(A))=C_0$  and \eqref{equ:A-condition-original},
  there exists a unique convex viscosity solution $u\in C^0(\mathbb R^n)$ of \eqref{equ:Translated}.
\end{theorem}
\begin{proof}
 By Lemmas \ref{Lem:subsolution-0}, \ref{Lem:Subsolution}  and the computation at the start part of this section, for any $c\geq u_0$ we have a generalized symmetric subsolution $\underline u(x):=U(\frac{1}{2}x^TAx)$ with
  $$
  \underline u(0)=u_0\quad\text{and}\quad \underline u(x)=\frac{1}{2}x^TAx+c+O(|x|^{2-\delta_0})
  $$
  as $|x|\rightarrow\infty$.
  By a direct computation, $\underline u$ is convex near origin and satisfies
  $G_{\tau}(\lambda(D^2\underline u))\geq C_0$ in $\mathbb R^n\setminus\{0\}$.
  By
  Theorem \ref{thm:2.2},
  $$
  G_{\tau}(\lambda(D^2\underline u))\geq C_0\quad\text{in }\mathbb R^n.
  $$

  On the other hand,
  $$
  \overline u(x):=\frac{1}{2}x^TAx+c
  $$
  is a smooth convex solution of $G_{\tau}(\lambda(D^2u))= C_0$ in $\mathbb R^n$.
  From the proof of Lemma \ref{Lem:Subsolution}, we have $\underline u\leq  \overline u$ in $\mathbb R^n$. In fact by \eqref{equ:temp-7},
  $$
  \underline u(x)-\left(\frac{1}{2}x^TAx+c\right)=U-s-u_0-\mu_{\alpha}=\int_{+\infty}^s(\psi(r,\alpha)-1)dr<0,
  $$
  where $\alpha\geq 1$ is the unique constant such that $c=u_0+\mu_{\alpha}$.
  Let
  $$
  u(x):=\sup\{v(x):G_{\tau}(\lambda (D^2v))\geq C_0,~v(0)=u_0~\text{and}~ \underline u\leq v\leq \overline u,\text{ in }\mathbb R^n\setminus\{0\} \}.
  $$
  By a direct computation, $u\in C^0(\mathbb R^n)$ is convex, $u(0)=u_0$ and $u(x)\rightarrow \frac{1}{2}x^TAx+c$ at infinity. By Perron's method (see for instance \cite{Bao-Li-Li-2014,Hitoshi-PerronsMethod}), $u(x)$ defined above is a viscosity solution of \eqref{equ:Translated}. The uniqueness follows similarly by maximum principle as in Lemma \ref{Lem:maximumPrinciple}.
\end{proof}

Eventually, we finish this section by proving that the solution  found satisfies
\eqref{equ:Entire} and \eqref{equ:c-Geq-U0}. In fact, by equation $G_{\tau}(\lambda(D^2u))= C_0$ in $\mathbb R^n\setminus\{0\}$ and convexity $D^2u>0$, Theorem \ref{thm:2.2} implies $u$ is a subsolution on entire $\mathbb R^n$. From the proof of Theorem \ref{Thm:temp}, we have
$$
u(x)\leq \overline u(x)=\frac{1}{2}x^TAx+c\quad\text{in }\mathbb R^n.
$$
Hence \eqref{equ:c-Geq-U0} follows immediately.

\section{Proof of Proposition \ref{prop:no-genearlizedSol}}\label{Sec:NonExistence}

In this section, we prove Proposition \ref{prop:no-genearlizedSol}
following the line of Bao-Li-Li \cite{Bao-Li-Li-2014}. Let $U,s(x)$ be as in \eqref{equ:def-s}.

For $\tau=0$ case, the result follows from the classification results by Jin-Xiong \cite{Jin-Xiong}.

For $\tau\in(0,\frac{\pi}{4})$ case, we only need to prove the following result for $G_{\tau}$ operator and desired result follows
 by eigen-decomposition \eqref{equ:eigendecomposition-2}.
\begin{proposition}
  For any $0<s_1<s_2<\infty$, if there exists a generalized symmetric classical convex solution $u(x)=U(s)$ of
  \begin{equation}\label{equ:Gtau-cut}
  G_{\tau}(\lambda(D^2u))=C_0\quad\text{in }\left\{
  x:\frac{1}{2}x^TAx\in(s_1,s_2)
  \right\},
  \end{equation}
  where $A=\mathtt{diag}(a_1,\cdots,a_n)>0$ satisfies $G_{\tau}(\lambda(A))= C_0$. Then either
  \begin{equation}\label{equ:classify-nonexist}
  a_1=a_2=\cdots=a_n\quad\text{or}\quad u=\frac{1}{2}x^TAx +c
  \end{equation}
  for some $c\in\mathbb R$.
\end{proposition}
\begin{proof}
  For any $s\in (s_1,s_2)$ and $i_0=1,\cdots,n$, take
  $x=(0,\cdots,0,\sqrt{\frac{2s}{a_{i_0}}},0,\cdots,0)$. By \eqref{equ:subsolu-1} and \eqref{equ:subsolu-2}, $u(x)$ is a classical solution of \eqref{equ:Gtau-cut} if
  $$
  \dfrac{(U')^n\sigma_n(a)+2sU''(U')^{n-1}}{
  \displaystyle\prod_{i=1}^n(U'+\frac{2b}{a_i})
  +2sU''\prod_{i=1,\cdots,n\atop i\not=i_0}(U'+\frac{2b}{a_i})
  }=c_0.
  $$
  Consequently,
  \begin{equation}\label{equ:temp-2}
  \begin{array}{llllll}
    & \displaystyle (U')^n+2sU''(U')^{n-1}-c_0\prod_{i=1}^n(U'+\frac{2b}{a_i})
    \\
    =&\displaystyle c_02sU''\prod_{i=1\cdots,n\atop i\not=i_0}(U'+\frac{2b}{a_i}).\\
  \end{array}
  \end{equation}

  If $U''\equiv 0$ in $(s_1,s_2)$, then $u=\frac{1}{2}x^TAx+c$ immediately.

  If $U''(\overline s)\not=0$ at a point $\overline s\in (s_1,s_2)$, since the left hand side of \eqref{equ:temp-2} is independent of choice of $i_0$, we have
  $$
  2\overline sU''(\overline s)\prod_{i=1\cdots,n\atop i\not=i_0}(U'(\overline s)+\frac{2b}{a_{i_0}})= C(\overline s)\quad\forall~i_0=1,2,\cdots,n.
  $$
  Thus in this case, $$
  a_1=a_2=\cdots=a_n.
  $$
\end{proof}

For $\tau=\frac{\pi}{4}$ case, similarly by eigen-decomposition, we only need to prove the following result.
\begin{proposition}
  For any $0<s_1<s_2<\infty$, if there exists a generalized symmetric classical convex solution $u(x)=U(s)$ of
  \begin{equation}\label{equ:Gtau-cut-2}
  G_{\tau}(\lambda(D^2u))=C_0\quad\text{in }\left\{
  x:\frac{1}{2}x^TAx\in(s_1,s_2)
  \right\},
  \end{equation}
  where $A=\mathtt{diag}(a_1,\cdots,a_n)>0$ satisfies $G_{\tau}(\lambda(A))= C_0$. Then \eqref{equ:classify-nonexist} holds for some $c\in\mathbb R$.
\end{proposition}
\begin{proof}
  For any $s\in (s_1,s_2)$ and $i_0=1,\cdots,n$, take
  $x=(0,\cdots,0,\sqrt{\frac{2s}{a_{i_0}}},0,\cdots,0)$. As a necessary condition for such solution exists, $C_0<0$.
  Since $u(x)$ is a classical solution of \eqref{equ:Gtau-cut-2}, we have
  $$
  c_0\cdot \sigma_n(\lambda(D^2u))=\sigma_{n-1}(\lambda(D^2u)),
  $$
  where $c_0=-\frac{\sqrt 2}{2}C_0>0.$
  By \eqref{equ:subsolu-1} and Proposition 1.2 of \cite{Bao-Li-Li-2014},
  $$
  \begin{array}{llll}
  &\displaystyle c_0\cdot (U')^n\sigma_n(a)+c_0\cdot U''(U')^{n-1}\sigma_n(a)\sum_{i=1}^na_ix_i^2\\
  =&\displaystyle
  (U')^{n-1}\sigma_{n-1}(a)
  +U''(U')^{n-2}\sum_{i=1}^n(a_ix_i)^2\sigma_{n-2;i}(a),\\
  \end{array}
  $$
  i.e.,
  \begin{equation}\label{equ:temp-9}
   c_0 (U')^n\sigma_n(a)+2c_0sU''(U')^{n-1}\sigma_n(a)
    -(U')^{n-1}\sigma_{n-1}(a)=
    2sU''(U')^{n-2}a_{i_0}\sigma_{n-2;i_0}(a).
  \end{equation}

  If $U''\equiv 0$ in $(s_1,s_2)$, then $U'=\frac{\sigma_{n-1}(a)}{c_0\sigma_n(a)}=1$ and $u=\frac{1}{2}x^TAx+c$ is a quadratic solution.

  If $U''(\overline s)\not\equiv 0$ at a point $\overline s\in(s_1,s_2)$, since the left hand side of \eqref{equ:temp-9} is independent of choice of $i_0$, we have
  $$
  2sU''(U')^{n-2}a_{i_0}\sigma_{n-2;i_0}(a)=C(\overline s)\quad\forall ~ i_0=1,2,\cdots,n.
  $$
  By the continuity of solution, we may assume without loss of generality that $U'(\overline s)\not=0$, otherwise there exists $\overline s'$ close to $\overline s$ such that $U''(\overline s')\not=0$ and $U'(\overline s')\not=0$. Together with the fact that
  $$
  \sigma_{n-1}(a)=\sigma_{n-1;i}(a)+a_i\sigma_{n-2;i}(a)\quad\forall ~ i=1,2,\cdots,n,
  $$
  we have $\sigma_{n-1;i_0}(a)$ is independent of the choice of $i_0$. Thus in this case, all $a_{i_0}$ are the same.
\end{proof}

For $\tau\in(\frac{\pi}{4},\frac{\pi}{2})$, the equations  are very similar to $\tau=\frac{\pi}{2}$ case and only need to prove for the latter case for simplicity.

\iffalse

If $A$ has non-positive eigenvalues, then there exists $x\in\mathbb R^n\setminus\{0\}$ such that $\frac{1}{2}x^TAx=0$. Hence in this case, the generalized symmetric solution of \eqref{equ:translatedequ} is a classical solution of
$$
G_{\tau}(\lambda(D^2u))=C_0\quad\text{in }\mathbb R^n.
$$
By the Bernstein-type results as in \cite{Yu.Yuan1,Yu.Yuan2}, $u$ is a quadratic polynomial and the proof is finished. Hence we only need to prove for  $A=\mathtt{diag}(a_1,\cdots,a_n)>0$ and $0<s_1<s_2<\infty$.

\fi

\begin{proposition}
  For any $0<s_1<s_2<\infty$, if there exists a generalized symmetric classical convex solution $u(x)=U(s)$ of
  \begin{equation}\label{equ:Gtau-cut-3}
  \sum_{i=1}^n\arctan\lambda_i(D^2u)=C_0\quad\text{in }\left\{
  x:\frac{1}{2}x^TAx\in(s_1,s_2)
  \right\},
  \end{equation}
  where $A=\mathtt{diag}(a_1,\cdots,a_n)>0$ satisfies $\sum_{i=1}^n\arctan a_i=  C_0$. Then \eqref{equ:classify-nonexist} holds up to some constant $c$.
\end{proposition}
\begin{proof}
  By the algebraic form of \eqref{equ:Gtau-cut-3} (see for instance \cite{CNS-DirichletIII}), $\lambda(D^2u)$ satisfies
  \begin{equation*}
\cos C_0 \sum_{0 \leq 2 k+1 \leq n}(-1)^{k} \sigma_{2 k+1}\left(\lambda\left(D^{2} u\right)\right)-\sin C_0 \sum_{0 \leq 2 k \leq n}(-1)^{k} \sigma_{2 k}\left(\lambda\left(D^{2} u\right)\right)=0.
\end{equation*}
  For any $s\in (s_1,s_2)$ and $i_0=1,\cdots,n$, take
  $x=(0,\cdots,0,\sqrt{\frac{2s}{a_{i_0}}},0,\cdots,0)$.
  Since $u(x)$ is a classical solution of \eqref{equ:Gtau-cut-3},  by \eqref{equ:subsolu-1} and \eqref{equ:represent-sigmak} we have
  $$
  \begin{array}{lll}
  & \displaystyle \cos C_0\left(
  \sum_{0\leq 2k+1\leq n}(-1)^k\left((U')^{2k+1}\sigma_{2k+1}(a)+2sU''(U')^{2k}a_{i_0}\sigma_{2k;i_0}(a)\right)
  \right)\\
  =& \displaystyle \sin C_0\left(
\sum_{0\leq 2k\leq n}(-1)^k\left((U')^{2k}\sigma_{2k}(a)+2sU''(U')^{2k-1}a_{i_0}\sigma_{2k-1;i_0}(a)\right)
  \right),
  \end{array}
  $$
  i.e.,
  \begin{equation}\label{equ:temp-10}
  \sum_{k=0}^{n-2}2s p_k U''(U')^ka_{i_0}\sigma_{k;i_0}(a)=
  \left(
  \begin{array}{llllll}
  2s p_1 U''\\
  2s p_2 U''U'\\
   2s p_3 U''(U')^2\\
  \vdots\\
  2s p_{n-2} U''(U')^{n-2}\\
  \end{array}
  \right)^T\cdot
  \left(
  \begin{array}{llll}
   a_{i_0}\\
   a_{i_0}\sigma_{1;i_0}(a)\\
   \vdots\\
   a_{i_0}\sigma_{n-2;i_0}(a)\\
  \end{array}
  \right)
  =G(s),
  \end{equation}
  where
  $$
  p_k=(-1)^{k}\cos \left(C_0+\left(\frac{k}{2}-\left[\frac{k}{2}\right]\right)\pi\right)
  $$
  for $k=0,1,\cdots,n-1$
  and
  $$
  \begin{array}{llll}
  G(s)&=&\displaystyle -\cos C_0\left(
  \sum_{0\leq 2k+1\leq n}(-1)^k(U')^{2k+1}\sigma_{2k+1}(a)
  \right)\\
  &&\displaystyle+\sin C_0\left(
  \sum_{0\leq 2k\leq n}(-1)^k(U')^{2k}\sigma_{2k}(a)
  \right)\\
  &&\displaystyle- 2s p_{n-1}U''(U')^{n-1}\sigma_n(a).\\
  \end{array}
  $$
  Since $\sin C_0=0$ and $\cos C_0=0$ cannot hold at the same time, we may only consider the non-zero part of \eqref{equ:temp-10} and the following proof still works. Hereinafter, we only prove for $\sin C_0\cdot\cos C_0\not=0$ case.

  If $U''\equiv 0$ in $(s_1,s_2)$, then $U'\equiv 1$ and $u=\frac{1}{2}x^TAx+c$ is a quadratic polynomial.

  If $U''(\overline s)\not\equiv 0$ at a point $\overline s\in(s_1,s_2)$,
  then by the continuity of solution, we may assume without loss of generality that $U''(s)\not=0$ and $U'(s)\not=0$ in $(s_3,s_4)\subset (s_1,s_2)$. Pick $n-1$ points such that $$
  s_3<r_1<r_2<\cdots<r_{n-1}<s_4,\quad r_i\not=0
  $$
  and we prove that the matrix formed by coefficients of \eqref{equ:temp-10} is invertible, i.e.,
  \begin{equation}\label{equ:temp-11}
  \det \left(
  \begin{array}{ccccc}
    2r_1 p_1U''(r_1) & \cdots &  2r_1 p_{n-2}U''(r_1)(U'(r_1))^{n-2}\\
    2r_2 p_1U''(r_2) &   \cdots &  2r_2  p_{n-2}U''(r_2)(U'(r_2))^{n-2}\\
    \vdots & \ddots & \vdots\\
    2r_n p_1U''(r_{n-1}) &  \cdots &
   2r_{n-1} p_{n-2}U''(r_{n-1})(U'(r_{n-1}))^{n-2}\\
  \end{array}
  \right)\not=0.
  \end{equation}
  By \eqref{equ:temp-10} and \eqref{equ:temp-11},  $(a_{i_0},\cdots,a_{i_0}\sigma_{n-1;i_0}(a))$ is a constant vector independent of $i_0$. Thus in this case, all $a_{i_0}$ are the same.

  To prove \eqref{equ:temp-11}, since $\sin C_0\cdot\cos C_0\not=0$, $0<r_1<\cdots<r_{n-1}$, $U''(s)\not=0$ and $U'(s)\not=0$, by the determinate of Vandermonde form matrix, we have
  $$
  \det \left(
  \begin{array}{cccccc}
    1 & U'(r_1) & \cdots & (U'(r_1))^{n-2}\\
    1 & U'(r_2) & \cdots & (U'(r_2))^{n-2}\\
    \vdots & \vdots & \ddots & \vdots\\
    1 & U'(r_{n-1}) & \cdots & (U'(r_{n-1}))^{n-2}\\
  \end{array}
  \right)=\prod_{i>j\atop
  i,j=1,2,\cdots,n-1}(U'(r_i)-U'(r_j))\not=0.
  $$
  Hence \eqref{equ:temp-11} is proved and the result follows immediately.
\end{proof}

\small

\bibliographystyle{plain}

\bibliography{IsolatedSingularity}

\bigskip

\noindent Z.Liu \& J. Bao

\medskip

\noindent  School of Mathematical Sciences, Beijing Normal University\\
Laboratory of Mathematics and Complex Systems, Ministry of Education\\
Beijing 100875, China \\[1mm]
Email: \textsf{liuzixiao@mail.bnu.edu.cn}\\[1mm]
Email: \textsf{jgbao@bnu.edu.cn}

%\bigskip

%\noindent J. Bao, J. Xiong \& Z. Zhou

%\medskip

%\noindent  School of Mathematical Sciences, Beijing Normal University\\
%Laboratory of Mathematics and Complex Systems, Ministry of Education\\
%Beijing 100875, China \\[1mm]
%Email: \textsf{jgbao@bnu.edu.cn, jx@bnu.edu.cn, zhouziwei@mail.bnu.edu.cn}

\end{document}